\documentclass[joc]{ipart}

\pubyear{2014}
\volume{0}
\issue{0}
\firstpage{1}
\lastpage{25}
\arxiv{1310.4093}

\usepackage{times}
\usepackage{amssymb, amsthm, amsmath}
\usepackage{graphicx,xspace}
\usepackage{epsfig}
\usepackage{enumerate}
\usepackage{enumitem}
\usepackage{xcolor}
\usepackage{tikz}
\usepackage{gensymb}
\usepackage[T1]{fontenc}
\usepackage[utf8]{inputenc}
\usepackage{bm}
\usepackage{ifpdf}
\usepackage{tikz}
\usepackage{todonotes}

\usetikzlibrary{shapes}
\usetikzlibrary{shapes.geometric}
\usetikzlibrary{shapes.symbols}
\usetikzlibrary{trees}

\colorlet{darkgreen}{green!50!black} 
\colorlet{violet}{blue!50!red}
\colorlet{orange}{yellow!50!red!80!black}
\tikzstyle{root}=[diamond, draw,  minimum size=4mm, inner sep=0]
\tikzstyle{part2}=[regular polygon, regular polygon sides=4, minimum size=6mm,draw=black,text=black, inner sep=0]
\tikzstyle{part2tiny}=[regular polygon, regular polygon sides=4, minimum size=1.5mm,draw=black,text=black, inner sep=0]
\tikzstyle{part1}=[circle, draw=black, text=black, minimum size=4mm, inner sep=.2mm]
\tikzstyle{part3}=[regular polygon, regular polygon sides=3, minimum size=4mm, draw, inner sep=0]
\tikzstyle{part4}=[regular polygon, regular polygon sides=5, draw, inner sep=0]
\tikzstyle{every child}=[level distance=8mm]

\theoremstyle{plain}

\newtheorem{lemma}{Lemma}[section]
\newtheorem{theorem}[lemma]{Theorem}

\newtheorem{proposition}[lemma]{Proposition}

\newtheorem{definition}[lemma]{Definition}
\newtheorem{construction}[lemma]{Construction}

\theoremstyle{remark}
\newtheorem*{remark}{Remark}
\newtheorem*{example}{Example}

\newcommand{\G}[2]{G_#1(#2)} 

\newcommand{\N}{\mathbb{N}}

\newcommand{\mcE}{\mathcal{E}}
\newcommand{\mcU}{\mathcal{U}}
\newcommand{\mcB}{\mathcal{B}}

\newcommand{\mcS}{\mathcal{S}}
\newcommand{\mcC}{\mathcal{C}}
\newcommand{\mcD}{\mathcal{D}}

\newcommand{\Id}{\mathrm{Id}}

\newcommand{\h}{\mathfrak{h}}
\newcommand{\al}{\alpha}
\newcommand{\om}{\omega}
\newcommand{\si}{\sigma}
\newcommand{\be}{\beta}
\newcommand{\ka}{\kappa}

\newcommand{\ld}{\ldots}
\newcommand{\cd}{\cdots}

\newcommand{\spl}{\mathrm{splice}}

\def\gb{\overline{g}}
\def\psipi{\psi_\pi}

\DeclareMathOperator{\wt}{wt}

\begin{document}
\begin{frontmatter}

\title{An edge-weighted hook formula for labelled trees}
\runtitle{An edge-weighted hook formula for labelled trees}

\author{\fnms{Valentin} \snm{Féray}\thanksref{t1}\corref{}\ead[label=e1]{valentin.feray@math.uzh.ch}}
\thankstext{t1}{VF is partially supported by ANR projet PSYCO and SNF grant "Dual combinatorics of Jack polynomials".}
\address{Institüt für Mathematik, Universität Zürich, Wintherturerstrasse 190,\\
 8057 Zürich, Switzerland\\ \printead{e1}}

\author{\fnms{I.P.} \snm{Goulden}\thanksref{t2}\ead[label=e2]{ipgoulden@uwaterloo.ca}}
\thankstext{t2}{IPG is supported by a Discovery Grant from NSERC.}
\address{Dept. Combinatorics \& Optimization, University of Waterloo, Waterloo,\\ Ontario, Canada N2L 3G1 \\ \printead{e2}}

\and
\author{\fnms{Alain} \snm{Lascoux}}
\address{Regretfully deceased during the preparation of this article.}

\runauthor{V.~Féray, I.P. Goulden and A. Lascoux}

\begin{abstract}
A number of hook formulas and hook summation formulas have previously appeared, involving various classes of trees. One of these classes of trees is rooted trees with labelled vertices, in which the labels increase along every chain from the root vertex to a leaf. In this paper we give a new hook summation formula for these (\emph{unordered increasing}) trees, by introducing a new set of indeterminates indexed by pairs of vertices, that we call \emph{edge weights}. This new result generalizes a previous result by F\'eray and Goulden, that arose in the context of representations of the symmetric group via the study of Kerov's character polynomials. Our proof is by means of a combinatorial bijection that is a generalization of the Pr\"ufer code for labelled trees.
\end{abstract}

\begin{keyword}
hook formula, tree enumeration, combinatorial bijection, generating function
\end{keyword}


\received{\smonth{11} \sday{14}, \syear{2013}}


\end{frontmatter}

 \section{Introduction}
 
 \subsection{Background}
 
The classical hook formula of Frame, Robinson and Thrall \cite[Theorem 1]{HookLengthFormula1954} gives the simple ratio 
\[ \chi^{\lambda}(\Id_{|\lambda |})  = \frac{|\lambda|!}{\prod_{\Box \in \lambda} h(\Box)} \]
for the dimension $\chi^{\lambda}(\Id_{|\lambda |})$ of the irreducible representation of the symmetric group associated with the Young diagram $\lambda$. Here $|\lambda|$ is the number of boxes in the diagram and $h(\Box)$ is the size of the hook attached to the box $\Box$. This result is equivalent to an enumerative result, since it is also the number of labellings of the boxes of $\lambda$ with the elements of $\N_{|\lambda|}=\{ 1,\ldots ,|\lambda|\}$ (once each) so that the labels increase along each row, and down each column.

Many results that look similar have appeared since, and are commonly referred to as \emph{hook formulas}.
A number of these involve various classes of trees.
Let us fix some terminology.
A (unordered) {\em tree} is an acyclic connected graph. The vertex-set (or label-set) of a tree $T$ is denoted by $V(T)$. {\em Rooted} means that we distinguish a vertex; then each edge
can be oriented towards the root and we call the head and tail of the edge {\em father} and {\em son}, respectively. We denote the father of a vertex $v$ in a rooted tree $T$ by $f_T(v)$, and set $f_T(v)= 0$ when $v$ is the root vertex. Then the rooted tree $T$ is completely defined by giving $f_T(v)$ for all vertices $v$ (in particular,
unless specified differently, sons of a given vertex are not ordered).
The {\em descendants} of a vertex are defined recursively as the sons and the descendants of the sons. If $u$ is a descendant of $v$, then we say that $v$ is an \emph{ancestor} of $u$.
The \emph{hook} attached to the vertex $v$ in the tree $T$, denoted by $\h_T(v)$, is the set consisting of $v$ and its descendants; the size of the hook $\h_T(v)$ is denoted by $h_T(v)$. An \emph{increasing} labelling of a rooted tree is a labelling of the vertices with distinct integers, so that the label of a son is always bigger than the label of its father; thus the root always gets the minimum label, and the labels increase along each branch from the root. An \emph{increasing} tree is an increasing labelling of the rooted tree.

It is well-known that
the number of ordered\footnote{Here, {\em ordered} means that labellings 
  $\begin{array}{c}
        \begin{tikzpicture}
            [font=\tiny,scale=.4]
            \node [part2tiny] {$1$}
                child{node [part2tiny] {$2$}
		}
                child{node [part2tiny] {$3$}
                }
            ;
        \end{tikzpicture}
\end{array}$ and
 $\begin{array}{c}
        \begin{tikzpicture}
            [font=\tiny,scale=.4]
            \node [part2tiny] {$1$}
                child{node [part2tiny] {$3$}
		}
                child{node [part2tiny] {$2$}
                }
            ;
        \end{tikzpicture}
\end{array}$ 
must be counted as different labellings.} increasing labellings of a given rooted tree is given by a formula that
looks like Frame-Robinson-Thrall formula.
Namely,
D. Knuth \cite[§5.1.4 Exer. 20]{KnuthArtProg3} proved that the number $L(T)$ of ordered increasing labellings of a rooted tree $T$ with vertex-set $\N_{|T|}$  is given by
\begin{equation}\label{EqHookOneTree}
    L(T) = \frac{|T|!}{\prod_{v \in T} h_T(v)},
\end{equation}
where $|T|$ is the number of vertices of $T$.

Another type of hook formula is a \emph{hook summation formula}. For example, let $\mcB_r$ denote the set of rooted binary trees with $r$ vertices (as usual for binary trees, sons of a given vertex are ordered). There is a well-known one-to-one correspondence between increasing binary trees with vertex-set $\N_r$, and permutations of size $r$ (see {\em e.g.} \cite[p. 23-25]{StanleyEC1}). Combining this with the rooted tree hook result~\eqref{EqHookOneTree}, and dividing by $r!$, yields the summation formula
\begin{equation}
    \sum_{ T\in\mcB_r}
    \;\prod_{v \in T} \frac{1}{h_T(v)}=1.
    \label{EqHookSum}
\end{equation}

More details on these hook formulas and some related works 
can be found in~\cite{FerayGoulden}.
In this article, two of us gave a hook summation formula that involved unordered increasing trees, which means that the sons of a vertex are not ordered. For our summation formula, we use the following notation for falling factorials: $(a)_m =a(a-1)\cdots (a-m+1)$ for positive integers $m$, with $(a)_0=1$, and $(a)_m=1/(a-m)_{-m}$ for negative integers $m$. Let $r\geq 1$ be an integer and $x_1, \cdots,x_r$ be formal variables. Let $\mcU_r$ denote the set of unordered increasing trees with vertex-set $\N_r$, and for $T\in\mcU_r$, define a weight $\wt(T)$ by
    \[ \wt(T) = \prod_{v=2}^r x_{f_T(v)}
            \bigg( \Big(\sum_{u \in \h_T(v)} x_u \Big) - h_T(v) + 1 \bigg).\]
 Then our hook summation formula~\cite{FerayGoulden} was given by
    \begin{equation}\label{hookform}
    \sum_{T\in\mcU_r} \wt(T)
    = x_1 \cdots x_r
    (x_1+\cdots +x_r-1)_{r-2}.
    \end{equation}
    \label{ThmHookFormula}
Three proofs of this result were presented in~\cite{FerayGoulden}. One of these involved Kerov's character polynomials (see, \textit{e.g.},\cite{Biane2003}), and thus gives a connection to the representation theory of the symmetric group, that does not seem related to the Frame-Robinson-Thrall formula.  

We also proved that~\eqref{hookform} specializes to a classical enumerative formula for Cayley trees. A \emph{Cayley tree} is a tree with labelled vertices (so they are distinguishable) -- these are not embedded in the plane, and there is no root vertex. Let $\mcC_r$ denote the set of Cayley trees with vertex-set $\N_r$. Borchardt~\cite{BorchardtCayley} and Cayley \cite{CayleyTrees}
proved that, for $r\geq 1$,
\begin{equation}
    \label{EqCayleyRefined}
    \sum_{T\in\mcC_r}
 x_1^{d_T(1)} \cdots x_r^{d_T(r)} =
x_1 \cdots x_r (x_1+\cdots +x_r)^{r-2},
\end{equation}
where $d_T(i)$ denotes the degree of the vertex $i$ in the tree $T$. We proved in~\cite{FerayGoulden} that~\eqref{hookform} specializes to~\eqref{EqCayleyRefined} in the case $x_1,\cdots,x_r \to \infty$, that is
for the highest degree terms in the $x_i$. On the right-hand sides, this is straightforward, so the work here is on the left-hand sides, for which we constructed a combinatorial mapping between the sets $\mcU_r$ and $\mcC_r$.

\subsection{The main result}

In this paper, we prove a new hook summation formula for unordered increasing trees. This formula is given in the following Theorem, which is our main result. This generalizes~\eqref{hookform} by introducing a set of doubly indexed indeterminates that we will refer to as \emph{edge weights}.
\begin{theorem}\label{NewHook}
Let $r\geq 2$ be an integer and $x_i$, $i=1,\ldots ,r$, $y_{i,j}$, $2\le i \le  j\le r$ be formal variables.
    For an unordered increasing tree $T$ with vertex-set $\N_r$,
    define the weight to be
    \[ \wt_y(T) = \prod_{v=2}^r \Big(  x_{f_T(v)} \sum_{u \in \h_T(v)} y_{v,u} \Big) .\]
    Then
    \begin{equation}\label{yhookform}
    \sum_{T\in\mcU_r} \wt_y(T)
    = x_1 y_{r,r} \prod_{i=2}^{r-1} \Bigg(  \sum_{j=1}^i x_j y_{i,i} + \sum_{j=i+1}^{r} x_i y_{i,j} \Bigg) .
    \end{equation}
\end{theorem}

Note that Theorem~\ref{NewHook} specializes to~\eqref{hookform} immediately, by the following substitution: $y_{v,u}=x_u-1$ for $v<u$, and $y_{u,u}=x_u$.

Our proof of Theorem~\ref{NewHook} is by a combinatorial bijection. This involves a number of stages, and in our description, it will be convenient to identify the left-hand and right-hand sides of~\eqref{yhookform} separately, as
\begin{equation}\label{LRHSmain}
L(x,y)=  \sum_{T\in\mcU_r} \wt_y(T) ,\quad R(x,y)= x_1 y_{r,r} \prod_{i=2}^{r-1} \Bigg(  \sum_{j=1}^i x_j y_{i,i} + \sum_{j=i+1}^{r} x_i y_{i,j} \Bigg) .
\end{equation}
Of course, in these terms, our main result is equivalent to 
\begin{equation}\label{mainequiv}
L(x,y)=R(x,y).
\end{equation}

\subsection{Outline of paper}
In the remainder of this paper we give a combinatorial proof of our main result. This is carried out by defining a combinatorial mapping in Section 2 that we describe in terms of an operation on unordered increasing trees called \emph{splice}. Then in Section 3 we prove a number of properties of our splice operation, enabling us to prove that the combinatorial mapping is a bijection. This directly proves~\eqref{mainequiv}, and hence Theorem~\ref{NewHook}.

There is one intriguing aspect of our main result that we have been unable to resolve. Note that our proof of the main result in this paper is based on a bijection for~$\mcU_r$, the set of \emph{unordered increasing trees}. However, if we evaluate the right-hand side of ~\eqref{yhookform} at $x_i=1$ for all $i$, and $y_{i,j}=1$ for all $i,j$, then we obtain $r^{r-2}$. But as we have noted above, $|\mcC_r|=r^{r-2}$~(\cite{BorchardtCayley},~\cite{CayleyTrees}), which suggests that there should be a combinatorial proof of the main result based on a bijection for~$\mcC_r$, the set of \emph{Cayley trees}. We have been unable to find such a proof, and suggest it as a problem for others to resolve.

\section{A combinatorial mapping}

\subsection{Dominating functions}
For a set $\mcS$ of positive integers, let $\Pi(\mcS)$ denote the set of partitions of $\mcS$ into an unordered set of nonempty subsets. The subsets are called the \emph{blocks} of the partition, and we denote the number of blocks of a partition $\pi$ by $|\pi|$. If $\pi$ has blocks $\pi_1,\ld ,\pi_k$, then we let $\mu_i= \max\,\pi_i$, for $i=1,\ld ,k$, and we index the blocks so that $\mu_1<\cd <\mu_k$.

For two sets $\mcS$ and $\mcS^{\prime}$ of positive integers,
 the function $g:\mcS\rightarrow\mcS'$ is called \emph{dominating} if $g(i)\geq i$ for all $i\in\mcS$. 
For such a function $g$, we denote
\[ \wt_g= \prod_{i \in S} y_{i,g(i)}.\]
We say that a dominating function $g: \mathcal{S} \rightarrow \mathcal{S}$ (\textit{i.e.}, with $\mathcal{S}^{\prime} =\mathcal{S}$) is "on $\mathcal{S}$". Consider the functional digraph of a dominating function $g$ on $\mcS$: the vertices are the elements of $\mcS$, and the directed edges are given by $(i,g(i))$, $i\in\mcS$. The vertex-sets of the connected components (ignoring the directions on edges) form a partition $\pi\in\Pi(\mcS)$, and we say that $g$ has \emph{induced} partition $\pi$. Let $\mcD(\pi)$ denote the set of all dominating functions on $\mcS$ with induced partition $\pi$, and let
\begin{equation*}
    D(\pi)=\sum_{g\in\mcD(\pi)} \wt_g
    =\sum_{g\in\mcD(\pi)}\;\;\prod_{i\in\mcS}y_{i,g(i)}.
\end{equation*}

For $\pi\in\Pi(\mcS)$, let $\mcE(\pi)$ denote the set of unordered increasing trees $T$ on vertex-set $\mcS$ such that every block of $\pi$ is a subchain of $T$.
In other words, for every pair of elements $i<j$ in the same block of $\pi$, $i$ is an ancestor of $j$ in $T$.
For any unordered increasing tree $T$, let 
\begin{equation}\label{kadef}
\kappa(T)=\prod_{i\in V(T)}\; x_i^{\si_i(T)},
\end{equation}
where $\si_i(T)$ denotes the number of sons of vertex $i$ in $T$.

Now we consider a restricted class of set partitions. If $\mcS$ is a set of positive integers containing $1$, then $\Pi_1(\mcS)$ is the set of partitions of $\mcS$ in which $\{ 1\}$ is a block. In this case, necessarily $\pi_1=\{ 1\}$, and $\mu_1=1$. For such a  partition $\pi\in\Pi_1(\mcS)$, let
\begin{equation*}
\mcC(\pi)=\N_{\mu_2}\times\cd\times \N_{\mu_{|\pi|-1}}=\{ (c_2,\ld,c_{|\pi|-1}):1\leq c_i\leq \mu_i,i=2,\ld ,|\pi |-1\},
\end{equation*}
(if $|\pi|= 2$, the set $\mcC(\pi)$ contains one element: the empty list). For $c\in\mcC(\pi)$, let
\begin{equation}\label{omdef}
\om(c,\pi)=\frac{x_{c_2}\cd x_{c_{|\pi |-1}}}{x_{\mu_2}\cd x_{\mu_{|\pi |}}} \prod_{i\in\mcS}\; x_i.
\end{equation}

There is a close connection between dominating functions and the expressions $L(x,y)$, $R(x,y)$ defined in~\eqref{LRHSmain}, given in the following result.

\begin{proposition}\label{LRDconnection}
For any integer $r \ge 2$, one has:
\begin{align*}
\!\!\!\!\!\!\!\! &\text{\emph{(a)}} \qquad\qquad\qquad\qquad  y_{1,1} \cdot L(x,y) &=&  \quad\sum_{\pi\in\Pi_1(\N_r)} D(\pi)\sum_{T\in\mcE(\pi)}\kappa(T), \\
\!\!\!\!\!\!\!\! &\text{\emph{(b)}} \qquad\qquad\qquad\qquad  y_{1,1} \cdot R(x,y) &=& \quad\sum_{\pi\in\Pi_1(\N_r)} D(\pi)\sum_{c\in\mcC(\pi)} \om(c,\pi).
\end{align*}
\end{proposition}

\begin{proof}
(a) 
For this equation, by definition,
\[L(x,y)=  \sum_{T\in\mcU_r} \wt_y(T).\]
But $\wt_y(T)=\kappa(T) \sum_g \wt_g$, where the sum runs over 
functions $g$ from $\{2,\ldots,r\}$ to $\{2,\ldots,r\}$ such that,
for each $u$, its image $g(u)$ lies in $\h_T(u)$
(as $T$ is an increasing tree, such functions are automatically dominating).
We can extend such functions $g$ to $\overline{g} : \N_r \to \N_r$
by setting $\overline{g}(1)=1$. Note that $\wt_{\overline{g}}=y_{1,1} \wt_g$.
The conditions $\overline{g}(1)=1$ and $g(u) \in \h_T(u)$ for $u \ge 2$
are equivalent to $\pi \in \Pi_1(\N_r)$ and $T \in \mcE(\pi)$,
where $\pi$ is the partition induced by $\overline{g}$.
Therefore
\begin{align*}
    y_{1,1} L(x,y) &=  \sum_{T\in\mcU_r} \kappa(T) 
    \left( \sum_{\pi \in \Pi_1(\mcS) \atop \text{s.t. }T \in \mcE(\pi)}
\sum_{\gb \in \mcD(\pi)} \wt_{\gb} \right) \\
&= \sum_{\pi \in \Pi_1(\mcS)} D(\pi) \sum_{T\in\mcE(\pi)} \kappa(T) ,
\end{align*}
giving part~(a) of the result.

(b) For this equation, note that $R(x,y)$ can be rewritten as
\[R(x,y)= x_1 x_2 \cdots x_{r-1} y_{r,r}
\prod_{i=2}^{r-1} \left( y_{i,i} \, \frac{x_1 + \cdots +x_i}{x_i} +\sum_{j=i+1}^r y_{i,j} \right).\]
Expanding the product in terms of dominating functions, we get
\[R(x,y) = x_1 x_2 \cdots x_{r-1} y_{r,r}
\sum_{g : \{2,\ldots,r-1\} \to \{2,\ldots,r\} \atop g \text{ dominating}}
\left( \wt_g \prod_{i\ :\ g(i)=i} \frac{x_1+\cdots+x_i}{x_i} \right).\]
As above, we can extend $g$ to $\gb : \N_r \to \N_r$
by setting $\gb(1)=1$ and $\gb(r)=r$.
Then $\wt_{\gb}=y_{1,1} y_{r,r} \wt_g$.
Now let $\pi$ denote the partition induced by $\gb$.
An important remark is that
the integers $i \neq 1,r$ such that $g(i)=\gb(i)=i$
are exactly the maxima of the blocks of $\pi$
except for $1$ and $r$, which are given by $\mu_2,\cdots,\mu_{|\pi|-1}$. Note also that $\mu_{|\pi|}=r$.
Hence
\[\prod_{i\ :\ g(i)=i} \frac{x_1+\cdots+x_i}{x_i}
= \frac{1}{x_{\mu_2}\cdots x_{\mu_{|\pi|-1}}} 
\sum_{c \in \mcC(\pi)} x_{c_2} \cdots x_{c_{|\pi|-1}},\]
and so we obtain 
\[x_1 x_2 \cdots x_{r-1} \prod_{i\ :\ g(i)=i} \frac{x_1+\cdots+x_i}{x_i}
= \sum_{c \in \mcC(\pi)} \om(c,\pi).\]
Thus $R(x,y)$ is given by
\[y_{1,1} R(x,y) = \sum_{ {\gb : \N_r \to \N_r \atop \gb(1)=1,\gb(r)=r}
\atop \gb \text{ dominating}}
\wt_{\gb} \sum_{c \in \mcC(\pi)} \om(c,\pi).\]
Note that $\gb$ dominating implies $\gb(r)=r$ while the condition $\gb(1)=1$
means that the induced partition $\pi$ is in $\Pi_1(\N_r)$.
Hence, splitting the sum depending on the induced partition of $\gb$,
we obtain part~(b) of the result.
\end{proof}

Comparing Proposition~\ref{LRDconnection} with~\eqref{mainequiv}, we see that Theorem~\ref{NewHook} is implied by a bijection

\begin{equation}\label{mainbijn}
\psipi:\mcC (\pi) \rightarrow \mcE (\pi): c \mapsto T,
\end{equation}
with the weight-preserving property that $\ka(T)=\om(c,\pi)$, for each $\pi\in\Pi_1(\N_r)$. We will find such a bijection $\psipi$.

\subsection{An operation on rooted trees}\label{subsec:splice}

A convenient construct for an unordered increasing tree $T$ with vertex $v$ is the $v$-decomposition of $T$, described as follows.

\begin{definition}
Let $T_1$ be an unordered increasing tree with root vertex $a_1$, and let $v$ be any vertex in $T$ ($v$ can be equal to $a_1$). Suppose that the unique maximal
chain from $a_1$ to $v$ is given by $a_1<a_2<\cd <a_k=v$, $k\geq 1$. Now remove the edges on the chain in $T$ from $a_1$ to $a_k$. There are $k$ components in the resulting graph, each of which is an unordered increasing tree, whose root vertex is on the chain from $a_1$ to $a_k$. Let $T^{(a_i)}$ be the component among these that is rooted at vertex $a_i$, $i=1,\ldots ,k$. Then the $v$-\emph{decomposition} of $T$ is the ordered list $T^{(a_1)},\ldots ,T^{(a_k)}$.
\end{definition}

An example of $v$-decomposition of a tree $T$ is given in Figure~\ref{fig:VDec}.

\begin{figure}
    \[\begin{array}{c|cccc}
        \begin{array}{c}
        \begin{tikzpicture}
            [font=\scriptsize,
            level 1/.style={sibling distance=1.5cm},
            level 2/.style={sibling distance=1cm},
            level 3/.style={sibling distance=1cm},
            level 4/.style={sibling distance=1cm},
            level 5/.style={sibling distance=1cm},
            level 6/.style={sibling distance=1cm}]
            \node [part2] {$1$}
                child{node [part2] {$2$}
                    child{node [part2] {$5$}
                    }
                    child{node [part2] {$8$}
                    }
                }
                child{node [part2] {$3$}
                    child{node [part2] {$9$}
                        child{node [part2] {$11$}
                            child{node [part2] {$12$}
                            }
                            child{node [part2] {$21$}
                            }
                        }
                        child{node [part2] {$15$}
                        }
                        child{node [part2] {\underline{$14$}}
                            child{node [part2] {$19$}
                                child{node [part2] {$20$}
                                }
                            }
                            child{node [part2] {$17$}
                            }
                        }
                    }
                }
            ;
            \node at (0,-4.8) {\normalsize $S$};
        \end{tikzpicture}
    \end{array}&
        \begin{array}{c}
        \begin{tikzpicture}
            [ font=\scriptsize,
            level 2/.style={sibling distance=1cm}]
            \node [part2] {$1$}
                child{node [part2] {$2$}
                    child{node [part2] {$5$}
                    }
                    child{node [part2] {$8$}
                    }
                };
                \node at (0,-2.4) {\normalsize $S^{(1)}$};
        \end{tikzpicture}
    \end{array}&
        \begin{array}{c}
        \begin{tikzpicture}
            [ font=\scriptsize]
            \node [part2] {$3$};
            \node at (0,-.8) {\normalsize $S^{(3)}$};
        \end{tikzpicture}
    \end{array}&
        \begin{array}{c}
        \begin{tikzpicture}
            [ font=\scriptsize,
            level 1/.style={sibling distance=1cm},
            level 2/.style={sibling distance=1cm}]
                \node [part2] {$9$}
                    child{node [part2] {$11$}
                        child{node [part2] {$12$}
                        }
                        child{node [part2] {$21$}
                        }
                    }
                    child{node [part2] {$15$}};
            \node at (0,-2.4) {\normalsize $S^{(9)}$};
        \end{tikzpicture}
    \end{array}&
        \begin{array}{c}
        \begin{tikzpicture}
           [ font=\scriptsize,
            level 1/.style={sibling distance=1cm}]
            \node [part2] {\underline{$14$}}
                child{node [part2] {$19$}
                    child{node [part2] {$20$}
                    }
                }
                child{node [part2] {$17$}
                }
            ;
            \node at (0,-2.4) {\normalsize $S^{(14)}$};
        \end{tikzpicture}
    \end{array}
    \end{array}\]
\caption{The $v$-decomposition of a tree $S$ (with $v=14$).}
\label{fig:VDec}
\end{figure}
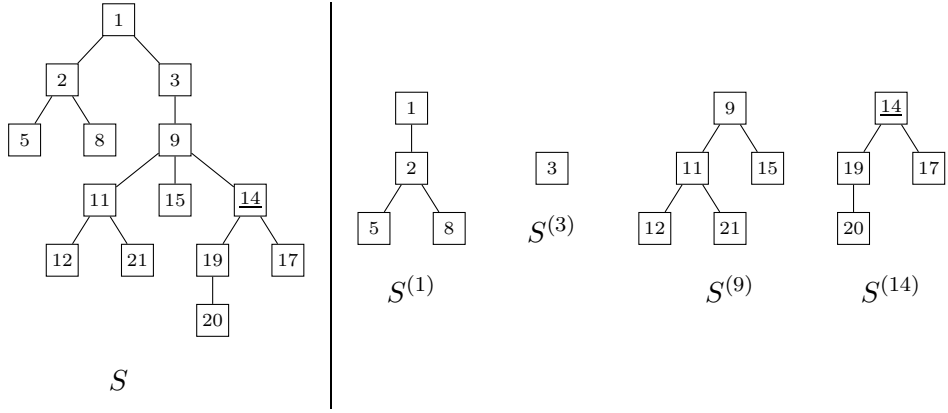
The bijection $\psipi$ will be constructed as the iteration of an elementary combinatorial operation on marked trees called splice,
that we define next, in terms of $v$-decompositions.

\begin{definition}\label{Def_Splice}
 Suppose that $T_1$ and $T_2$ are two unordered increasing trees with disjoint vertex-sets, and let $v_1,v_2$ be vertices in $T_1,T_2$ respectively, with $v_1>v_2$. Let the $v_1$-decomposition of $T_1$ be $T_1^{(a_1)},\ldots ,T_1^{(a_k)}$, and the $v_2$-decomposition of $T_2$ be $T_2^{(b_1)},\ldots ,T_2^{(b_m)}$. Then, since $a_k=v_1 >v_2=b_m$, we have 
$$ b_1<\cd <b_{\be_1}<a_1<\cd <a_{\al_1}<b_{\be_1+1}<\cd <b_{\be_2}$$
$$ <a_{\al_1+1}<\cd <a_{\al_2}< \cd <b_{\be_{\ell -1}+1}<\cd <b_{\be_{\ell}}<a_{\al_{\ell -1}+1}<\cd <a_{\al_\ell},  $$
for some unique $\ell \geq 1$ and $1\leq \al_1<\cd <\al_{\ell}=k$, $\; 0\leq \be_1<\cd <\be_{\ell}=m$. 

Then define the \emph{splice} of $T_1$ and $T_2$, with \emph{splicing} vertices $v_1$ and $v_2$, denoted by
\begin{equation}\label{splicedef}
R = \spl(T_1,v_1;T_2,v_2),
\end{equation}
to be the unordered increasing tree with $v_1$-decomposition given by
\begin{equation}\label{spldecomp}
T_2^{(b_1)}, \ld ,T_2^{(b_{\be_1})},T_1^{(a_1)},\ld ,T_1^{(a_{\al_1})},\ld,T_2^{(b_{\be_{\ell -1}+1})},\ld ,T_2^{(b_{\be_{\ell}})},
\end{equation}
$$T_1^{(a_{\al_{\ell -1}+1})},\ld ,T_1^{(a_{\al_\ell})}.$$
\end{definition}

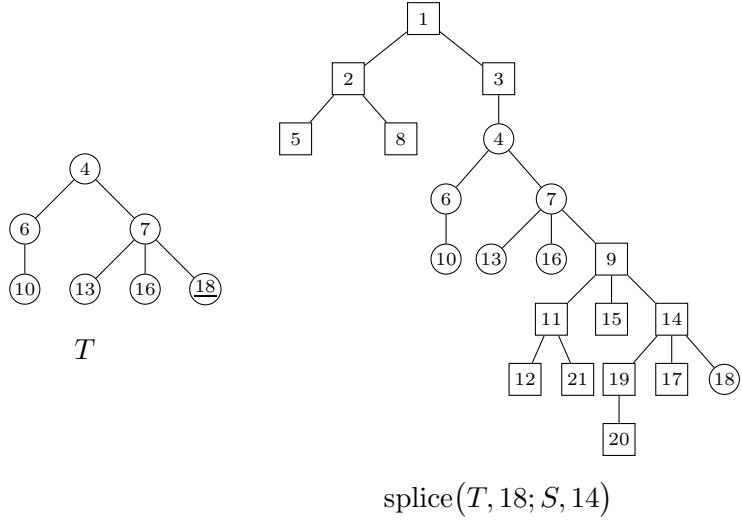
\begin{figure}
    \[\begin{tikzpicture}
        \begin{scope}[level 1/.style={sibling distance=1.6cm},
            level 2/.style={sibling distance=8mm}]
            \node [part1] {\scriptsize $4$}
                child{node [part1] {\scriptsize $6$}
                    child{node [part1] {\scriptsize $10$}
                    }
                }
                child{node [part1] {\scriptsize $7$}
                    child{node [part1] {\scriptsize $13$}
                    }
                    child{node [part1] {\scriptsize $16$}
                    }
                    child{node [part1] {\underline{\scriptsize $18$}}
                    }
                }
            ;
            \node at (0,-2.4) {\normalsize $T$};
        \end{scope}
        \begin{scope}
            [xshift=4.5cm, yshift=2cm, font=\scriptsize,
            level 1/.style={sibling distance=2cm},
            level 2/.style={sibling distance=1.4cm},
            level 3/.style={sibling distance=1.4cm},
            level 4/.style={sibling distance=.8cm},
            level 5/.style={sibling distance=.8cm},
            level 6/.style={sibling distance=.7cm},
            level 7/.style={sibling distance=.7cm}]
            \node [part2] {$1$}
                child{node [part2] {$2$}
                    child{node [part2] {$5$}
                    }
                    child{node [part2] {$8$}
                    }
                }
                child{node [part2] {$3$}
                    child{node [part1] {$4$}
                        child{node [part1] {$6$}
                            child{node [part1] {$10$}
                            }
                        }
                        child{node [part1] {$7$}
                            child{node [part1] {$13$}
                            }
                            child{node [part1] {$16$}
                            }
                            child{node [part2] {$9$}
                                child{node [part2] {$11$}
                                    child{node [part2] {$12$}
                                    }
                                    child{node [part2] {$21$}
                                    }
                                }
                                child{node [part2] {$15$}
                                }
                                child{node [part2] {$14$}
                                    child{node [part2] {$19$}
                                        child{node [part2] {$20$}
                                        }
                                    }
                                    child{node [part2] {$17$}
                                    }
                                    child{node [part1] {$18$}
                                    }
                                }
                            }
                        }
                    }
                }
            ;
            \node at (1,-6.4) {\normalsize $\spl\big( T,18;S,14 \big)$};
        \end{scope}
    \end{tikzpicture}\]
    \caption{The \emph{splice} of two trees ($S$ is given in Figure~\ref{fig:VDec}; \emph{splicing} vertices are underlined).}
\label{egsplice}
\end{figure}

An example of the splice operation is given in Figure~\ref{egsplice}.

Note in the construction of $R$ above that the vertex-set of $R$ is the (disjoint) union of the vertex-sets of $T_1$ and $T_2$.
Also, if $\be_1\ge 1$, equivalent to $b_1<a_1$, then the root vertex of R is $b_1$, the same as the root vertex of $T_2$. However, if $\be_1=0$, equivalent to $a_1<b_1$, then the root vertex of R is $a_1$, the same as the root vertex of $T_1$.

In the following result we record some simple but important properties of the splice operation.

\begin{lemma}\label{inheritancestor}
Suppose that $R = \spl(T_1,v_1;T_2,v_2)$. 
\begin{itemize}
\item[\emph{(a)}]
For $\ell = 1,2$, and any pair of vertices $i$ and $j$  in $T_{\ell}$, then $i$ is an ancestor of $j$ in $T_{\ell}$ if and only if $i$ is an ancestor of $j$ in $R$.
Consequently, given the sets $V(T_1)$ and $V(T_2)$,
one can recover $T_1$ and $T_2$ from $R$.
\item[\emph{(b)}]
For all vertices $i$ in $T_1$, we have $\si_i(R)=\si_i(T_1)$. For the vertex $v_2$ in $T_2$, we have $\si_{v_2}(R)=\si_{v_2}(T_2)+1$, but for all other vertices $i$ in $T_2$, we have $\si_i(R)=\si_i(T_2)$.
\end{itemize}
\end{lemma}

\begin{proof}
(a) For $\ell=1$, consider the $v_1$-decomposition of $T_1$: $T_1^{(a_1)},\ldots ,T_1^{(a_k)}$.
Let $I$ and $J$ denote the indices such that $i \in T_1^{(a_I)}$ and $j \in T_1^{(a_J)}$.
\begin{itemize}
    \item If $I=J$ then $i$ is an ancestor of $j$ in $T_1$ if and only if $i$ is an ancestor of $j$ in $T_1^{(a_I)}$.
        The same is true in $R$.
    \item If $I \neq J$, then $i$ is an ancestor of $j$ in $T_1$ if and only if $I<J$ and $i=a_I$.
        The same is true in $R$.
\end{itemize}
In both cases, we see that $i$ is an ancestor of $j$ in $T_1$
if and only if $i$ is an ancestor of $j$ in $R$.
This ends the proof of part~(a) for $\ell=1$. The case $\ell=2$ is similar.

%
%
%
%

(b) This is immediate.
\end{proof}

\subsection{A candidate for our bijection.}

In this section we describe a mapping $\psipi$ that we claim is a suitable bijection for~\eqref{mainbijn}. 
To describe $\psipi$, consider a set partition $\pi\in\Pi_1(\N_r)$ with $|\pi|=k$, so $\pi$ has blocks $\pi_1=\{ 1\},\ld ,\pi_k$, and $1=\mu_1<\cd <\mu_k$ where $\mu_i$ is the largest element of $\pi_i$, for $i=1,\ld ,k$. Recall that $r\geq 2$ (and hence $k\geq 2$). Consider also a $(k-2)$-tuple $c=(c_2,\ldots ,c_{k-1})\in\mcC(\pi)$.  We apply an iterative procedure in which we have a forest of unordered increasing trees on vertex-set $\N_r$ at every stage, and we apply \emph{splice} to reduce the number of components by one between successive stages. 

\begin{construction}\label{mapping}
Initially, at Stage 1, we have the forest with components $\tau_1, \ldots ,$ $\tau_k$, where $\tau_{\ell}$ is the increasing chain whose vertices are the elements of the set $\pi_{\ell}$, $\ell =1,\ld ,k$ (so $\tau_1$ consists of the single vertex $1$). At every stage we also keep track of an integer $\nu$ in $[r]$, with $\nu =1$ initially. Then, at Stage $i$, for $i=2,\ldots ,k-1$, we input a forest with components $\tau_1$, $\tau_i, \tau_{i+1},\ldots ,\tau_k$, together with an integer $\nu$, and create the following output:
Iwahori-Heckesp
\begin{description}
\item[Case 1.]
If $c_i$ is a vertex in $\tau_1$ or $\tau_i$, then set $\tau_1 = \spl(\tau_i,\mu_i;\tau_1,\nu)$, omit $\tau_i$, and set $\nu=c_i$;
\item[Case 2.]
If $c_i$ is a vertex in $\tau_j$ for some $j>i$, then set $\tau_j= \spl(\tau_i,\mu_i;\tau_j,c_i)$, and omit $\tau_i$.
\end{description}
After completion of the above procedure, we are left with a pair of increasing rooted trees $\tau_1$ and $\tau_k$, and an integer $\nu$. Then finally, at Stage $k$, we let
\begin{equation}\label{psidef}
\psipi(c)=T, \qquad\mbox{where}\quad T=\spl(\tau_k,\mu_k;\tau_1,\nu).
\end{equation}
\end{construction}

\begin{remark}
During the construction, $\nu$ is always a vertex of $\tau_1$ and moreover $\nu \leq \mu_i$ after Stage $i$.
Hence the splice $\spl(\tau_i,\mu_i;\tau_1,\nu)$ is well-defined in Case 1. 
Similarly, since $c_i \le \mu_i$, the splice in Case 2 is well-defined 
(we cannot have $c_i = \mu_i$ in Case 2, as this would imply that $c_i$ is a vertex of $\tau_i$).
\end{remark}

An example of the mapping $\psipi$ is given in Figure~\ref{egpsi} where $r=9$ and $k=4$. 
At Stage 2, we applied Case 1 because $c_2=4$ was a vertex of $\tau_2$, while, at Stage 3, we applied Case 2 because $c_3=5$ was a vertex of $\tau_4$.
At each stage, the value of $\nu$ is recorded by an edge on $\tau_1$ without child.
\begin{figure}[t]
    \[\begin{array}{c|c}
        \begin{array}{c}
            \text{Initially: Stage 1} \vspace{2mm} \\
        \begin{tikzpicture}[font=\scriptsize]
            \node [root] {$1$} child;
            \node at (1,0) [part1] {$3$}
                child{ node [part1] {$4$}
                    child{node [part1] {$6$}
                    }
                }
            ;
            \node at (2,0) [part2] {$2$}
                child{ node [part2] {$7$}
                    child{node [part2] {$8$}
                    }
                }
            ;
            \node at (3,0) [part3] {$5$}
                child{ node [part3] {$9$}
                }
            ;
            \node at (0,-2.2){\normalsize $\tau_1$};
            \node at (1,-2.2){\normalsize $\tau_2$};
            \node at (2,-2.2){\normalsize $\tau_3$};
            \node at (3,-2.2){\normalsize $\tau_4$};
        \end{tikzpicture}
    \end{array}
    &
    \begin{array}{c}
        \text{Stage 2 ($c_2=4$)} \vspace{2mm}\\
        \begin{tikzpicture}[font=\scriptsize,
            level 3/.style={sibling distance=9mm}]
            \node at (.5,0) [root] {$1$} 
                child {node [part1] {$3$}
                    child{ node [part1] {$4$} 
                        child{node [part1] {$6$}
                        }
                        child
                    }
                }
            ;
            \node at (2,0) [part2] {$2$}
                child{ node [part2] {$7$}
                    child{node [part2] {$8$}
                    }
                }
            ;
            \node at (3,0) [part3] {$5$}
                child{ node [part3] {$9$}
                }
            ;
            \node at (0.5,-3){\normalsize $\tau_1$};
            \node at (2,-3){\normalsize $\tau_3$};
            \node at (3,-3){\normalsize $\tau_4$};
        \end{tikzpicture}
    \end{array}\\ \hline
    \begin{array}{c}
        \text{Stage 3 ($c_3=5$)} \vspace{2mm}\\
        \begin{tikzpicture}[font=\scriptsize,
            level 2/.style={sibling distance=9mm},
            level 3/.style={sibling distance=9mm}]
            \node [root] at (.5,0) {$1$} 
                child {node [part1] {$3$}
                    child{ node [part1] {$4$} 
                        child{node [part1] {$6$}
                        }
                        child
                    }
                }
            ;
            \node at (2.5,0) [part2] {$2$}
                child {node [part3] {$5$}
                    child{ node [part2] {$7$}
                        child{node [part2] {$8$}
                        }
                    }
                    child{ node [part3] {$9$}
                    }      
                }
            ;
            \node at (.5,-3){\normalsize $\tau_1$};
            \node at (2.5,-3){\normalsize $\tau_4$};
        \end{tikzpicture}
    \end{array}&
    \begin{array}{c}
            \vspace{-3mm} \hbox{ } \\
        \text{Finally: Stage 4} \vspace{2mm} \\
        \begin{tikzpicture}[font=\scriptsize,
            level 5/.style={sibling distance=9mm}]
            \node [root] at (.5,0) {$1$} 
                child{node [part2] {$2$}
                    child{node [part1] {$3$}
                        child{ node [part1] {$4$}
                            child{node [part1] {$6$}
                            }
                            child{node [part3] {$5$}
                                child{node [part2] {$7$}
                                    child{node [part2] {$8$}
                                    }
                                }
                                child{ node [part3] {$9$}
                                }
                            }
                        }
                    }
                }
            ;
            \node at (1.8,-4.8) {\normalsize $T$};
        \end{tikzpicture}
    \end{array}
    \end{array}\]
\caption{Applying $\psipi$ to $c=( 4, 5)$, when $\pi$ has blocks $\pi_1=\{ 1\}$, $\pi_2=\{3,4,6\}$, $\pi_3=\{ 2,7,8\}$, $\pi_4 =\{ 5,9\}$.}
\label{egpsi}
\end{figure}
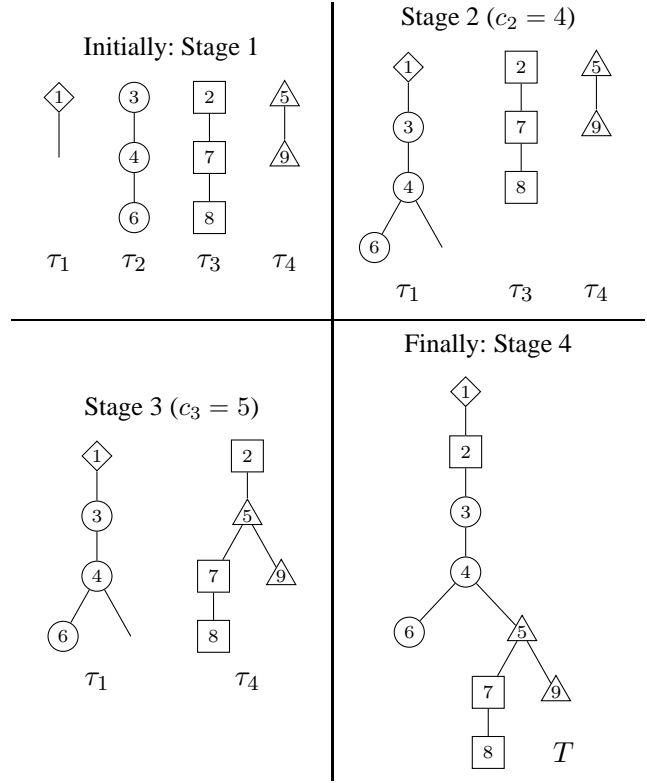

\begin{proposition}\label{psiimageok}
Given $r\ge 2$, a set partition $\pi\in\Pi_1(\N_r)$, and $c\in \mcC (\pi)$, suppose that $\psipi (c)=T$ is constructed as in~\eqref{psidef} above. Then
$$ T\in\mcE(\pi), \qquad\mbox{and}\qquad \ka(T)=\om(c,\pi).$$
\end{proposition}

\begin{proof}
By construction, it is clear that $T$ is a tree with vertex-set $\N_r$.
We have to check that it is indeed in $\mcE(\pi)$.
Initially, in the iterative procedure for $\psipi$, we have components $\tau_i$,  the chain consisting of the elements of the block $\pi_i$ of $\pi$, $i=1,\ld ,k$. The rooted tree $T$ is constructed by applying the splice operation $k-1$ times, to join the initial components together in some order. The fact that $T\in\mcE(\pi)$ now follows immediately from Lemma~\ref{inheritancestor}(a). 

Also, initially, we have 
$$\ka(\tau_1)\cd\ka(\tau_k)=\frac{1}{x_{\mu_2}\cd x_{\mu_{k}}} \prod_{i=2}^r\; x_i.$$
But for the terminating tree $T$, from Lemma~\ref{inheritancestor}(b), we have
$$\ka(T)=x_1x_{c_2}\cd x_{c_{k-1}}\ka(\tau_1)\cd\ka(\tau_k),$$
and the fact that $\ka(T)=\om(c,\pi)$ now follows immediately from~\eqref{omdef} (in this case we have $\mcS=\N_r$).
\end{proof}

Comparing Proposition~\ref{psiimageok} with~\eqref{mainbijn}, we see that the mapping $\psipi$ above is indeed a candidate for a bijective proof of our main result.

\begin{theorem}\label{thm_bijn}
For each $r\ge 2$ and $\pi\in\Pi_1(\N_r)$, the mapping
$$\psi_{\pi}: \mcC(\pi)\rightarrow\mcE(\pi)$$
is a bijection.
\end{theorem}

We will prove Theorem~\ref{thm_bijn} in the next Section, by determining the inverse of $\psipi$. In our development, we will find it convenient to use terms that distinguish between the different ways in which ``splice'' is applied in Construction~\ref{mapping} -- a splice that arises in Case 1 or in~\eqref{psidef} (the final stage) is called an \emph{internal} splice, whereas one that arises in Case 2 is called an \emph{external} splice.

\section{Inverting the combinatorial mapping and a bijective proof of the main result}

The goal of this section is to construct the inverse of $\psipi$ in order to show
that it is a bijection. Throughout this Section, $\pi$ is a fixed partition in $\Pi_1(\N_r)$.
As in the previous Section, the blocks of $\pi$ are denoted by $\pi_1,\pi_2,\ldots,\pi_k$,  and the maximum elements in these blocks are denoted by $\mu_1, \mu_2,\ldots ,\mu_k$. We assume as before that $\mu_1 < \mu_2 < \cdots < \mu_k$.
When it is convenient, we will also use the notation $\pi^x$ for the block of $\pi$ containing the element $x$.
This should not be confused with $\pi_i$, which denotes the $i$-th part of $\pi$.

Consider a set $\mcS$ that is a union of blocks of $\pi$ (we shall say that $\mcS$ is $\pi$-\emph{compatible}).
In other words, $\mcS=\cup_{i \in I} \pi_i$, for some index set $I$.
We denote by $\pi_{|\mcS}$ the partition $\{ \pi_i, i\in I\}$ of $\mcS$. A tree $T\in\mcE (\pi_{|\mcS})$ for some $\pi$-compatible $S$ is said to be $\pi$-\emph{increasing}.

\subsection{Dependence graphs and irreducibility}

We begin by defining a directed graph associated with the $v$-decomposition of an unordered increasing tree.

\begin{definition}\label{def_depend}
Consider a $\pi$-increasing tree $T$, with ($\pi$-compatible) vertex-set $\mcS=\cup_{i \in I} \pi_i$. For a vertex $v$ in $T$, suppose that the $v$-decomposition of $T$ is given by
\begin{equation*}
T^{(a_1)},\ld ,T^{(a_\ell)},
\end{equation*}
where $a_{\ell}=v$. Then the $v$-\emph{dependence} graph of $T$, denoted by $\G{v}{T}$, is a directed graph with the following vertices and directed edges:
\begin{itemize}
    \item the vertex-set of $G_v(T)$ is $\{ \pi_i, i\in I\}$;
    \item for the directed edges of $G_v(T)$, consider, if any, a maximum element $\mu_i$, $i\in I$, which is not contained in the
        chain $a_1,\cdots,a_\ell$ of $T$. 
       This vertex belongs (as a nonroot vertex) to $T^{(a_j)}$ for some $j=1,\ldots ,\ell$.
        Then there is a directed edge from $\pi_i$ to $\pi^{a_j}$, for each such $i\in I$.
\end{itemize}
\end{definition}

For example, fix
\[ \pi = \big\{ \{1\},\{3,4,6\}, \{2,7,8\},\{5,9\} \big\}, \]
as in Figure \ref{egpsi}.
Then $G_9(T)$, the $9$-dependence graph of the tree $T$, obtained
in the final stage of Figure \ref{egpsi}, is drawn in Figure~\ref{fig:ex_dep}.

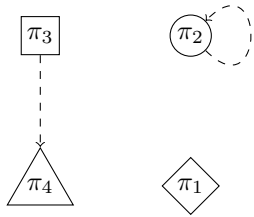
\begin{figure}
\begin{tikzpicture}
    \node[part1, inner sep=.5mm] (p1) at (2,0) {$\pi_2$ }; 
    \node[part2] (p2) at (0,0) {$\pi_3$ };
    \node[part3] (p3) at (0,-2) {$\pi_4$ };
    \node[root, inner sep=.5mm] at (2,-2) {$\pi_1$ };
    \draw[dashed,->]  (p2) -- (p3);
    \draw[dashed,->] (p1) {} .. controls +(1,-1) and +(1,1) .. (p1);
\end{tikzpicture}
\caption{Example of the $v$-dependence graph of a tree}
\label{fig:ex_dep}
\end{figure}

Note that the graph $G_v(T)$ defined above depends strongly on the partition $\pi$, but since $\pi$ is fixed throughout the section, we have omitted it from the notation.

In the remainder of this paper, we will particularly consider $v$-dependence graphs for the case in which $v$ is the vertex of maximum label. This motivates the following definition.

\begin{definition}\label{Def_irred}
A $\pi$-increasing tree $T$ with maximum vertex $v$ is called \emph{irreducible} if $G_v(T)$ is connected. Otherwise, $T$ is called \emph{reducible}.
\end{definition}
\begin{remark}
As we shall see later (Proposition~\ref{prop_Treducible}), any $\pi$-increasing tree with vertex set $[r]$ is reducible.
Therefore, the notion of irreducibility is interesting only for trees with smaller vertex sets.
\end{remark}

We now determine the form of $G_v(T)$ when $v$ is the maximum vertex of an irreducible tree $T$.   

\begin{lemma}\label{lem_conntree}
If a $\pi$-increasing tree $T$ with maximum vertex $v$ is irreducible, then the $v$-dependency graph $G_v(T)$ is an indirected tree rooted at $\pi^v$
({\em i.e.} all edges are directed towards the root).
\end{lemma}

\begin{proof}
Let the $v$-decomposition of $T$ be given by $T^{(a_1)},\ld ,T^{(a_\ell)}$, where $a_{\ell}=v$, so $v$ appears as the root vertex in $T^{(a_\ell)}$. Then $v$, which is the maximum element in the block $\pi^v$, cannot appear as a nonroot vertex in any tree of the $v$-decomposition, and $\pi^v$ has outdegree $0$ in $G_v(T)$. But every other vertex has outdegree at most $1$, and since $G_v(T)$ is connected, it can only be a tree in which every other vertex has outdegree exactly $1$ (in a connected graph, the number of vertices minus the number of edges is at most $1$, and a difference of $1$ occurs only for trees).
\end{proof}

We now consider the effect on irreducibility of applying the splice operation. 

\begin{lemma}\label{lem:splice_reducible}
Suppose that $T_1$ and $T_2$ are $\pi$-increasing trees with disjoint vertex-sets, and that $T_1$ is irreducible. For $i=1,2$, let $m_i$ denote the maximum element in the vertex-set of $T_i$. Let $v_2$ be a vertex in $T_2$ with $v_2<m_1$, and let 
$$T=\spl\big( T_1,m_1;T_2,v_2\big).$$
    \begin{itemize}
        \item[\emph{(a)}] If $m_1>m_2$, then $T$ is reducible. 
        More precisely,
        \begin{equation}\label{eq:Splice_DIsjointUnionG}
         G_{m_1}(T) = G_{m_1}(T_1) \sqcup G_{v_2}(T_2). 
         \end{equation}
        \item[\emph{(b)}] If $m_1<m_2$ and $T_2$ is irreducible, then $T$ is irreducible.
    \end{itemize}
\end{lemma}

\begin{proof}
(a) From~\eqref{spldecomp}, the trees in the $m_1$-decomposition of $T$ 
are either trees in the $m_1$-decomposition of $T_1$ or trees in the $v_2$-decomposition of $T_2$.
Then equality~\eqref{eq:Splice_DIsjointUnionG} follows from Definition~\ref{def_depend}.
In particular, $G_{m_1}(T)$ is not connected, and since $m_1$ is the maximum vertex in $T$, we conclude that $T$ is reducible.

\noindent
(b) Suppose that the $m_1$-decomposition of $T_1$ is $T_1^{(a_1)},\ldots ,T_1^{(a_{\ell})}$, and the $v_2$-decomposition of $T_2$ is $T_2^{(b_1)},\ldots ,T_2^{(b_n)}$. We have $a_{\ell}=m_1$, and since $m_1$ is the maximum vertex in $T_1$, then $T_1^{(a_{\ell})}$ consists of the single root vertex $a_{\ell}$. Now $T_2$ is an increasing tree, so we have $b_1<\cdots <b_n=v_2$. Also, by hypothesis we have $v_2<m_1<m_2$, and we conclude that $m_2$ is not contained in the chain $b_1,\ldots ,b_n$, which means that $m_2$ appears as a nonroot vertex in $T_2^{(b_u)}$ for some $1\leq u \leq n$. In particular, if $a_w>b_u$, then $a_w$ and $m_2$ are in different branches below $b_u$ in the tree $T$, and $T_1^{(a_w)}$ is entirely included in the tree rooted at $b_u$ in the $m_2$-decomposition of $T$.

Now the vertices of $G_{m_2}(T)$ consist of the vertices of $G_{m_1}(T_1)$
together with the vertices of $G_{m_2}(T_2)$.

To describe the edges of $G_{m_2}(T)$, we shall first consider the 
$m_2$-decomposition of $T$.
It is obtained as follows:
\begin{itemize}
  \item start with the $m_2$-decomposition of $T_2$;
  \item for each $a_w <b_u$, add the tree $T_1^{(a_w)}$.
        Indeed, the vertex $a_w$ is in the chain from the root to $m_2$
        in $T$, and the tree rooted at $a_w$ in the $m_2$-decomposition of $T$
        is the same as in the $v_1$-decomposition of $T_1$;
    \item for each $a_w >b_u$, add all vertices of $T_1^{(a_w)}$ to the tree rooted at $b_u$.
        Indeed, $a_w$ and $m_2$ are in different branches 
        below $b_u$ in the tree $T$ (because $m_2$ is in $T_2^{(b_u)}$).
        Hence, $T_1^{(a_w)}$ is entirely included in
        the tree rooted at $b_u$ in the $m_2$-decomposition of $T$.
\end{itemize}
This is illustrated in Figure \ref{fig:SpliceIsIrreducible}.
In this Figure, the $m_1$-decomposition of $T$
(which is by definition the union of the $m_1$-decomposition of $T_1$
and the $v_2$-decomposition of $T_2$)
is represented with blue and red dashed lines.
The $m_2$-decomposition of $T$ is drawn with plain black lines.
Finally, we have used green dotted lines for 
the $m_2$-decomposition of $T_2$ 
(it should be understood that the tree rooted at $b_u$ in this decomposition
contains only the vertices of $T_2$ in the corresponding green
dotted shape and of course no vertices of $T_1$).

    \begin{figure}
          \centering
          \def\svgwidth{200pt}
          \ifpdf
              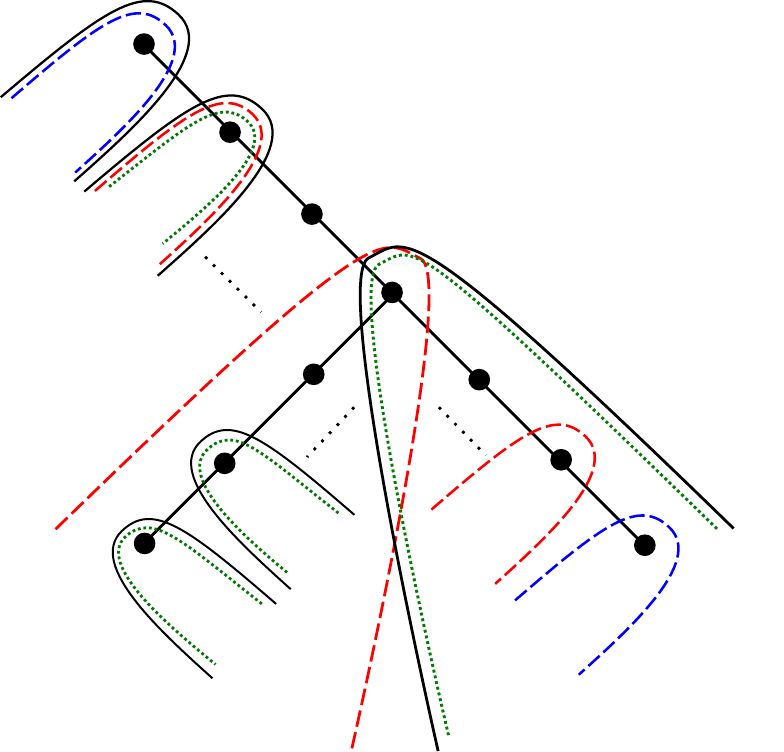
          \else
              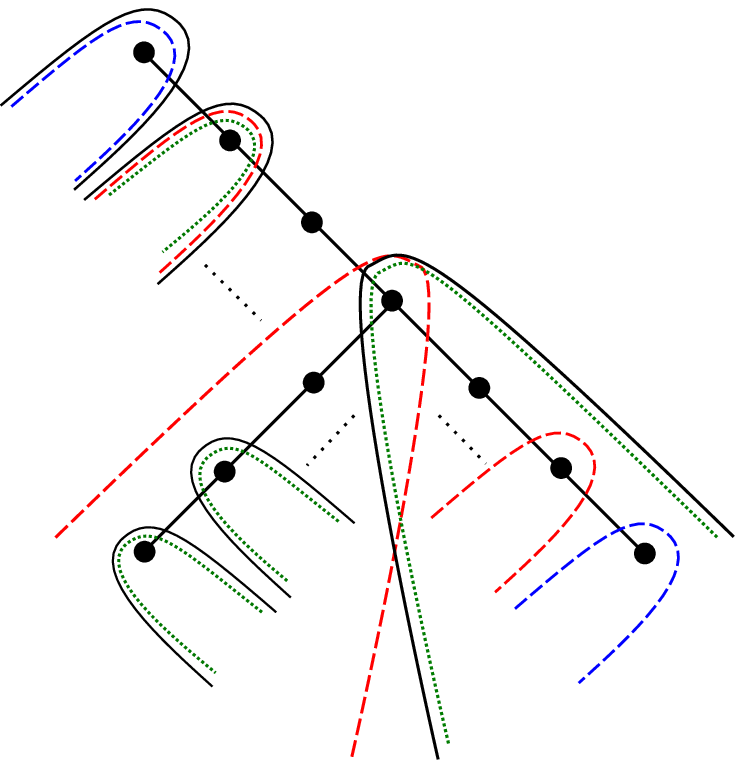
          \fi
          \caption{Illustration of the proof of Lemma \ref{lem:splice_reducible}}
      \label{fig:SpliceIsIrreducible}
    \end{figure}

From this, we can describe the edges of $G_{m_2}(T)$:
\begin{itemize}
    \item start with the edges of $G_{m_1}(T_1)$ together with the edges of $G_{m_2}(T_2)$;
    \item for any maximum element $\mu_i\neq m_1$ that appears as a vertex in $T_1^{(a_w)}$ for any $a_w>b_u$, remove the edge from $\pi_i$ to $\pi^{a_w}$ that appears in $G_{m_1}(T_1)$, and insert an edge from $\pi_i$ to $\pi^{b_u}$;
    \item finally, insert an edge from $\pi^{m_1}$ to $\pi^{b_u}$. 
\end{itemize}
Recall that $G_{m_1}(T_1)$ and $G_{m_2}(T_2)$ are both connected by hypothesis,
and we want to show that $G_{m_2}(T)$ is connected.
To do this, we will show that each vertex is in the connected component of $\pi^{b_u}$.
For vertices in $G_{m_2}(T_2)$, this is obvious as $G_{m_2}(T)$ contains all edges of $G_{m_2}(T_2)$.
From Lemma \ref{lem_conntree}, $G_{m_1}(T_1)$ is a directed tree of root $\pi^{m_1}$.
Then, for any vertex $v$ in $G_{m_1}(T_1)$  either the path from $v$ to $m_1$ is also in $G_{m_2}(T)$ and then the edge
from $\pi^{m_1}$ to $\pi^{b_u}$ proves that $v$ and $\pi^{b_u}$ are in the same connected component,
or this path is broken because one of its edge has been replaced by an edge to $\pi^{b_u}$.
In this case, the same conclusion that  $v$ and $\pi^{b_u}$ are in the same connected component holds.
Thus  $G_{m_2}(T)$ is connected, and since $m_2>m_1$, $m_2$ is the maximum vertex in $T$, so we conclude that $T$ is irreducible.
\end{proof}

\begin{example}
We give illustrations of Lemma~\ref{lem:splice_reducible} in both cases (a) and (b):

\noindent (a) Call $T_1$ and $T_2$ the trees $\tau_4$ and $\tau_1$ respectively 
of Stage 3 from Figure~\ref{egpsi}.
Then $m_1=9$ and we choose $v_2=4$.
The graph $G_9(T_1)$ has vertices $\{\pi_3,\pi_4\}$ and an edge from $\pi_3$ to $\pi_4$,
while $G_4(T_2)$ has vertices $\{\pi_1,\pi_2\}$ and an edge from $\pi_2$ to itself.
Then $T= \spl\big( T_1,m_1;T_2,v_2\big)$ is given in Stage 4 of Figure~\ref{egpsi} and its dependence graph $G_9(T)$,
drawn in Figure~\ref{fig:ex_dep}, is indeed the disjoint union of $G_9(T_1)$ and $G_4(T_2)$.

\noindent (b) Set $\pi=\big\{\{1\},\{3,6\},\{2,7\},\{9\},\{4,8,10\},\{5,11\} \big\}$
Consider the trees $T_1$ and $T_2$ from Figure~\ref{fig:ex_splice_irreducible} and 
choose $v_2=5$.
The tree $T=\spl\big( T_1,10;T_2,5\big)$ is also drawn in Figure~\ref{fig:ex_splice_irreducible}.
The corresponding dependence graphs are also given,
showing that $T$ is indeed irreducible.
Note that, as explained in our proof,
$G_{11}(T)$ differs from the disjoint union of $G_{10}(T_1)$ and  $G_{11}(T_2)$
as follows:
the edge from $\pi_4$ to $\pi_5$ in $G_{10}(T_1)$ is replaced 
in $G_{11}(T)$ by an edge from $\pi_4$ to $\pi_6$ ;
moreover, a new edge from $\pi_5$ to $\pi_6$ has also been added.
The graph $G_{11}(T)$ is determined using the $11$-decomposition of $T$,
given in Figure \ref{fig:11-dec}.

\begin{figure}[t]
    \[\hspace{-8mm}
    \begin{array}{c|c|c}
        T_1=\begin{array}{c}
        \begin{tikzpicture}
            [font=\scriptsize]
            \node [part3] {$2$}
                child{node [part1] {$4$}
                    child{node[part3] {$7$}}
                    child{node[part1] {$8$}
                        child{node[part2] {$9$}}
                        child{node[part1] {$10$}}
                    }
		        }
            ;
        \end{tikzpicture}
        \end{array}
        &
        T_2=\begin{array}{c}
        \begin{tikzpicture}
            [font=\scriptsize]
            \node [root] {$3$}
                child{node [part4] {$5$}
                    child{node[root] {$6$}}
                    child{node[part4] {$11$}}
		        }
            ;
        \end{tikzpicture}
        \end{array}
         & 
        T=\begin{array}{c}
        \begin{tikzpicture}
            [font=\scriptsize,level 4/.style={sibling distance=12mm},
            level 5/.style={sibling distance=8mm}]
            \node [part3] {$2$}
                child{node [root] {$3$}
                    child{node [part1] {$4$}
                        child{node[part3] {$7$}}
                        child{node [part4] {$5$}
                            child{node[root] {$6$}}
                            child{node[part4] {$11$}}
                            child{node[part1] {$8$}
                                child{node[part2] {$9$}}
                                child{node[part1] {$10$}}
                            }
                        }
                    }
		        }
            ;
        \end{tikzpicture}
        \end{array}
        \\ \hline 
        G_{10}(T_1)=\begin{array}{c}
        \begin{tikzpicture}
            [font=\scriptsize]
    \node[part3] (pr) at (0,0) {$\pi_3$ }; 
    \node[part1] (p1) at (0.75,.75) {$\pi_5$ };
    \node[part2] (p2) at (1.5,0) {$\pi_4$ };
    \draw[dashed,->]  (pr) -- (p1);
    \draw[dashed,->]  (p2) -- (p1);
        \end{tikzpicture}
        \end{array}
        &
        G_{11}(T_2)=\begin{array}{c}
        \begin{tikzpicture}
            [font=\scriptsize]
    \node[root] (p3) at (0.75,0) {$\pi_2$ }; 
    \node[part4] (p4) at (0,.75) {$\pi_6$ };
    \draw[dashed,->]  (p3) -- (p4);
        \end{tikzpicture}
        \end{array}
         & 
        G_{11}(T)=  \! \begin{array}{c}
        \begin{tikzpicture}
            [font=\scriptsize]
    \node[part3] (pr) at (0,0) {$\pi_3$ }; 
    \node[part1] (p1) at (0.75,.75) {$\pi_5$ };
    \node[part2] (p2) at (1.5,0) {$\pi_4$ };
    \node[root] (p3) at (3.5,0) {$\pi_2$ }; 
    \node[part4] (p4) at (2.75,.75) {$\pi_6$ };
    \draw[dashed,->]  (pr) -- (p1);
    \draw[dashed,->]  (p2) -- (p4);
    \draw[dashed,->]  (p1) -- (p4);
    \draw[dashed,->]  (p3) -- (p4);
        \end{tikzpicture}
        \end{array}
    \end{array}
    \]
    \caption{Splicing irreducible trees together with $m_1 < m_2$
    yields a new irreducible tree.}
    \label{fig:ex_splice_irreducible}
\end{figure}
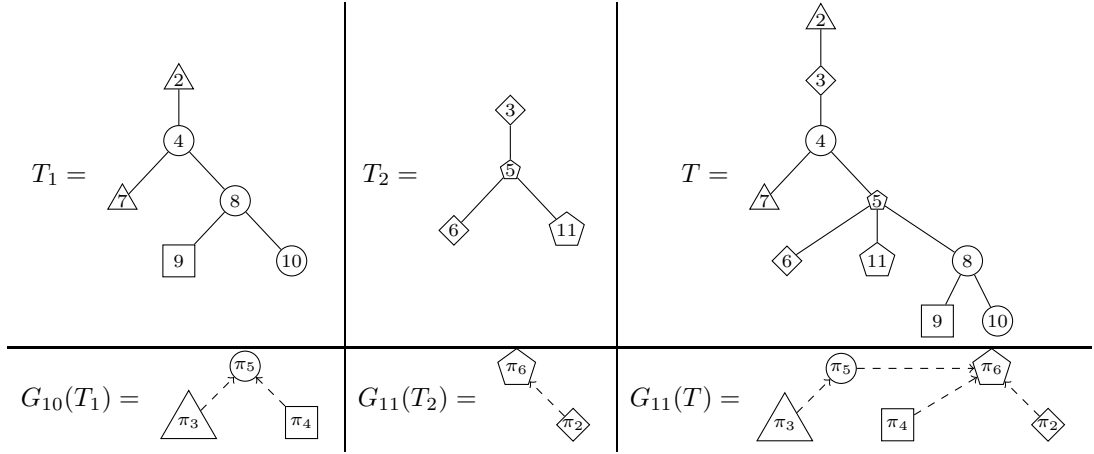
\end{example}

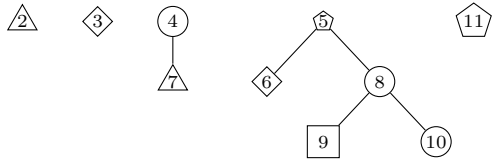
\begin{figure}[t]
    \[ \begin{tikzpicture}
            [font=\scriptsize]
            \node [part3] {$2$};
            \node at (1,0) [root] {$3$};
            \node at (2,0) [part1] {$4$}
                 child{node[part3] {$7$}};
            \node at (4,0) [part4] {$5$}
                 child{node[root] {$6$}}
                 child{node[part1] {$8$}
                      child{node[part2] {$9$}}
                      child{node[part1] {$10$}}
                 };
            \node[part4] at (6,0) {$11$};
    \end{tikzpicture}\]
    \caption{The $11$-decomposition of the tree $T$ from Figure~\ref{fig:ex_splice_irreducible}.}
    \label{fig:11-dec}
\end{figure}
The above result allows us now to classify the trees at every stage of Construction~\ref{mapping} by their irreducibility or reducibility.

\begin{proposition}\label{proposition_psiandred}

In Construction~\ref{mapping}:
\begin{itemize}
\item
At Stage 1, the trees $\tau_1,\ldots ,\tau_k$ are $\pi$-increasing and irreducible.
\item
For $i=2,\ldots ,k$, in the input to Stage $i$, the trees $\tau_i,\ldots ,\tau_k$ are $\pi$-increasing, irreducible, and contain $\mu_i,\ldots  \mu_k$, respectively; if $\tau_1$ has been created by applying an internal splice at some previous stage, then $\tau_1$ is $\pi$-increasing and reducible.
\item
The tree $\psi_{\pi}(c)=T$ is $\pi$-increasing and reducible.
\end{itemize}
\end{proposition}

\begin{proof}
At Stage 1, each tree $\tau_i$ consists of vertices of the single block $\pi_i$ of $\pi$, arranged as an increasing chain, so each $\tau_i$ is $\pi$-increasing. Also, since $G_{\mu_i}(\tau_i)$ has only the single vertex $\pi_i$, it is a connected graph, so $\tau_i$ is irreducible.

This result for Stage 1 serves as the base case for an induction on the stages, in which the remaining results follow immediately from Lemma~\ref{inheritancestor}(a) (for $\pi$-increasing), and Lemma~\ref{lem:splice_reducible} (part~(b) for irreducible, and~(a) for reducible).
\end{proof}

In particular, Proposition~\ref{proposition_psiandred} establishes that \emph{all} trees $T$ created as images of our combinatorial mapping $\psi_{\pi}$ are reducible. Thus, in order for $\psi_{\pi}:\mcC(\pi)\rightarrow\mcE(\pi)$ to be a bijection for any $\pi\in\Pi_1(\N_r)$, when $r\geq2$, it is necessary that all trees $T$ in $\mcE(\pi)$ are reducible. We prove that this is indeed the case in the following result. 
(note that
$T$ belongs to $\mcE(\pi)$ implies in particular that $T$ is $\pi$-increasing).

\begin{proposition}\label{prop_Treducible}
    Consider a $\pi$-increasing tree $T$ whose vertex-set contains $1$.
    We assume that the vertex-set of $T$ is not reduced to $\{1\}$
    and denote its maximum element by $M$.

    Then $\pi_1$ and $\pi^M$ are in different connected components of
    $G_M(T)$.
    In particular,
    then $T$ is reducible. 
\end{proposition}

\begin{proof}
Recall that the partition $\pi$ contains $\pi_1=\{ 1\}$ as a block, with maximum element $\mu_1=1$.
Now consider a $\pi$-increasing tree $T$ containing $1$
and its $M$-decomposition $T^{(a_1)},\ldots ,T^{(a_{\ell})}$,
with $a_{\ell}=M$ (where $M$ is the maximum vertex of $T$).
But vertex $1$ is the root vertex of every tree $T$ in $\mcE(\pi)$, so $a_1=1$.
Thus $\mu_1$ is the root vertex of $T^{(a_1)}$, and cannot appear as a nonroot vertex in any tree of the $r$-decomposition of $T$.

This implies that $\pi_1$ has outdegree $0$ in the $v$-dependence graph $G_v(T)$.
But $\pi^M$ has also outdegree $0$ (see the proof of Lemma~\ref{lem_conntree}),
and other vertices (if any) have outdegree at most $1$.
Therefore, $\pi_1$ and $\pi^M$ are in different connected components of          
    $G_M(T)$.
\end{proof}

\subsection{Irreducibility and inverting the combinatorial mapping}
In this section we prove that each application of \emph{splice} in our combinatorial mapping $\psi_{\pi}$ can be uniquely reversed by considering only the irreducibility or reducibility of the trees involved.

We begin with a simple condition for when a $\pi$-increasing tree can be written as the splice of two subtrees.

\begin{lemma}\label{Lem_splcondvtx}
Consider a $\pi$-increasing tree $T$, a $\pi$-compatible nonempty subset $V_1\subset V(T)$, and a vertex $v_1\in V_1$.  Then $T$ can be written as
\[T=\spl \big( T_1,v_1;T_2,v_2 \big) \]
for some $\pi$-increasing trees $T_1$ and $T_2$ with vertex-sets $V_1$ and $V_2 = V(T) \setminus V_1$
and for some vertex $v_2\in V_2$ if and only if
\[ V_1 = \bigcup_{i \in I} \pi_i,\]
where $\{\pi_i, i \in I\}$ is a union of vertex-sets of connected components of $G_{v_1}(T)$.

In this case, $T_1$, $T_2$ and $v_2$ are unique.
\end{lemma}


\begin{proof}
Given $T$, $v_1$ and $V(T_1)$, we immediately have $V_2 = V(T) \setminus V_1$. Then it is clear from Definition~\ref{Def_Splice} that $v_2$ is uniquely the first element of $V_2$ on the chain from $v_1$ to the root vertex of $T$ (note that $v_2<v_1$ since $T$ is an increasing tree). Moreover, the trees $T_1$ and $T_2$ themselves are then uniquely determined, since we know their $v_1$-decomposition and $v_2$-decomposition, respectively.

It only remains to determine conditions for $V_1$. Let the $v_1$-decomposition of $T$ be given by $T^{(a_1)},\ldots ,T^{(a_{\ell})}$, where $a_{\ell}=v_1$. Then from Definition~\ref{Def_Splice}, a necessary and sufficient condition  is that $V_1$ (and $V_2$) are unions of the $V(T^{(a_j)})$, $j=1,\ldots ,\ell$.
Since, by hypothesis, $V_1$ is a union of blocks of $\pi$, this is equivalent
to saying that $\pi$ is a union of blocks of the partition
\[ \Pi_1:= \pi \vee \left( \big\{ V(T^{(a_j)}), j=1,\ldots ,\ell \big\} \right). \]
It remains to see that this partition is nothing other than
\[ \Pi_2:=\{X_c, \ c\text{ connected component of }G_{v_1}(T)\},\text{ where }X_c=\bigcup_{\pi_i \in V_c} \pi_i.\]
To do this, take two partitions $\pi^x$ and $\pi^y$ which contain elements $x$ and $y$ in the same
set $V(T^{(a_j)})$. 
We want to prove that $\pi^x$ and $\pi^y$ are in the same connected components of $G_{v_1}(T)$.

Call $u$ the root of $T^{(a_j)}$
At least one of these elements, say $x$, is different from $u$.
Then by definition, there is an edge from $\pi^x$ to $\pi^u$ in $G_{v_1}(T)$. 
If $y=u$, there is an edge from $\pi^x$ to $\pi^y$, and thus they are in the same connected component of $G_{v_1}(T)$.
If $y \neq u$, the same argument as above implies that there is also an edge 
from $\pi^y$ to $\pi^u$, and one can also conclude that $\pi^x$ and $\pi^y$
lie in the same connected component of $G_{v_1}(T)$.
Hence $\Pi_1$ is finer than $\Pi_2$.

Conversely, suppose that there is an edge from $\pi_s$ to $\pi_t$ in $G_{v_1}(T)$.
Then this means that $\mu_s$ is a nonroot vertex in some $T^{(a_j)}$, 
with $a_j\in\pi_t$, which implies that there are elements of both $\pi_s$ and $\pi_t$ 
in the same subtree $T^{(a_j)}$.
Hence, $\Pi_2$ is finer than $\Pi_1$.

We conclude that $\Pi_1=\Pi_2$, which ends the proof of the Lemma.
\end{proof}

In the next result, which is the key to inverting $\psi_{\pi}$, we consider a $\pi$-increasing tree in which the vertex-set consists of two or more blocks of $\pi$. For such a tree with vertex-set $\mcS$ and maximum vertex $M$, we call $m=\max(\mcS \backslash \pi^M)$ the {\em second maximum} vertex.

\begin{lemma}\label{lemsplirred}
Suppose that $T$ is a $\pi$-increasing tree in which the vertex-set consists of two or more blocks of $\pi$, and let $M$ and $m$ be the maximum and second maximum vertices, respectively.
\begin{itemize}
\item[\emph{(a)}]
If $T$ is reducible, then it can be written uniquely as
$$T=\spl(T_1,M;T_2,t),$$
where $T_1$ and $T_2$ are $\pi$-increasing trees subject to:
\begin{itemize}
     \item[\textbullet]
$M$ is a vertex in $T_1$, $t$ is a  vertex in $T_2$,
     \item[\textbullet]
 $T_1$ is irreducible.
  \end{itemize}
 Moreover, if $1$ is a vertex of $T$,
  then it automatically belongs to $T_2$.
\item[\emph{(b)}]
If $T$ is  irreducible, then it can be written uniquely as
$$T=\spl(T_1,m;T_2,t),$$
where $T_1$ and $T_2$ are $\pi$-increasing trees subject to:
 \begin{itemize}
     \item[\textbullet]
$m$ is a vertex in $T_1$, $t$ and $M$ are vertices in $T_2$, with $t<m$,
     \item[\textbullet]
 $T_1$ and $T_2$ are irreducible.
  \end{itemize}
\end{itemize}
\end{lemma}

\begin{proof}
(a) From Lemma~\ref{Lem_splcondvtx}, the vertex-set $V_1$ of $T_1$ must correspond to a union of
connected components of $G_M(T)$.
In addition, from Lemma~\ref{lem:splice_reducible} (a),
$G_M(T_1)$ is the graph induced by $G_M(T)$ on $V_1$.
Hence, if we want $T_1$ to be irreducible, that is $G_M(T_1)$ to be connected,
then $V_1$ must correspond to a single connected component of $G_M(T)$.
Moreover, since we require $M$ to be in $T_1$, it must contain the block $\pi^M$.

Finally, $V_1$ is uniquely the vertex-set of the connected component of $G_M(T)$ containing $\pi^M$ (note that $t<M$ for all vertices $t$ in $T_2$, since $M$ is the maximum vertex in $T$). 
The result follows immediately from Lemma~\ref{Lem_splcondvtx}.

The property that, if $1$ is in $T$, then it is always in $T_2$ comes
from the fact that $\pi_1$ and $\pi^M$ are in different connected components of $G_M(T)$ 
(Proposition~\ref{prop_Treducible}) and the characterization of $V_1$ above.

(b) Consider the $M-$ and $m$-decompositions of $T$:
    \[T^{(a_1)}_M,\cdots,T^{(a_n)}_M \quad \text{and} \quad T^{(b_1)}_m,\cdots,T^{(b_\ell)}_m,\]
    in which $a_1=b_1$ is the root vertex of $T$, and $a_n=M$, $b_{\ell}=m$. Now $M>m$, so $m$ is not contained in the chain $b_1< \cdots < b_{\ell}$. Also, $T$ is irreducible, so Lemma~\ref{lem_conntree} with $v=M$ implies that $\pi^m$ has outdegree $1$ in $G_M(T)$.  But $m$ is the maximum element in $\pi^m$, so $m$ is a nonroot vertex in one of the trees in the $M$-decomposition of $T$, and hence $M$ is not contained in the chain $a_1<\cdots <a_n$. 
Thus  $j\ge 1$ exists so that $a_1=b_1,\ldots ,a_j=b_j$ and $a_{j+1}\neq b_{j+1}$, and $j<n$, $j<\ell$. Let $u=a_j=b_j$. The vertex $m$ lies in the subtree $T^{(u)}_M$, and before $T^{(u)}_M$, the $M$ and $m$-decompositions of $T$ coincide; see Figure~\ref{fig:Mm_dec} for the general picture
and Figures~\ref{fig:11-dec} and \ref{fig:10-dec}
 for a concrete example, in which $M=11$, and $m=10$.

    \begin{figure}
          \centering 
          \def\svgwidth{200pt}
          \ifpdf
                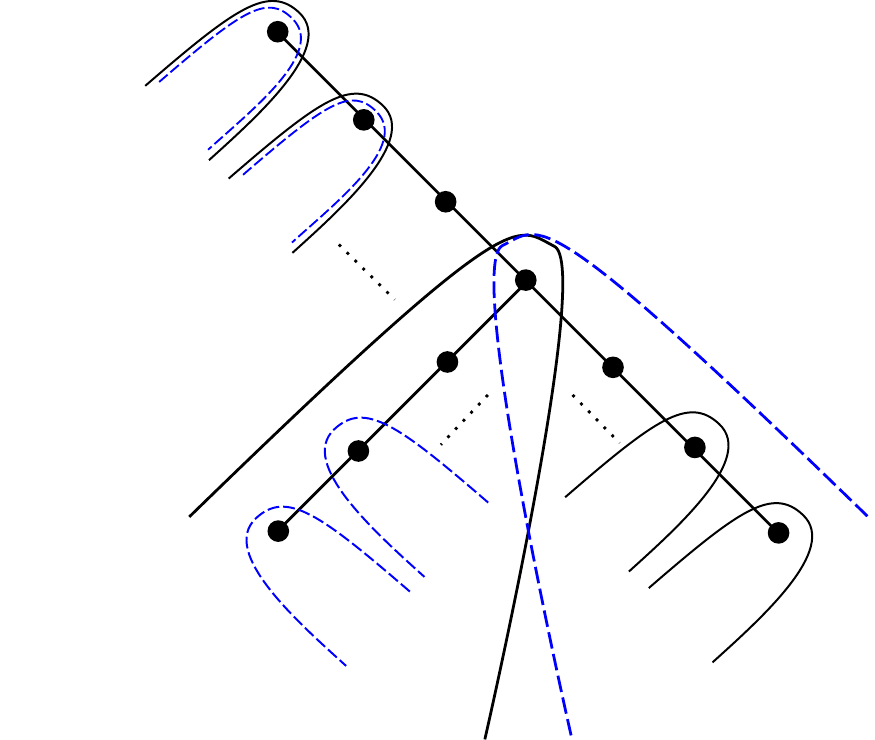
          \else
                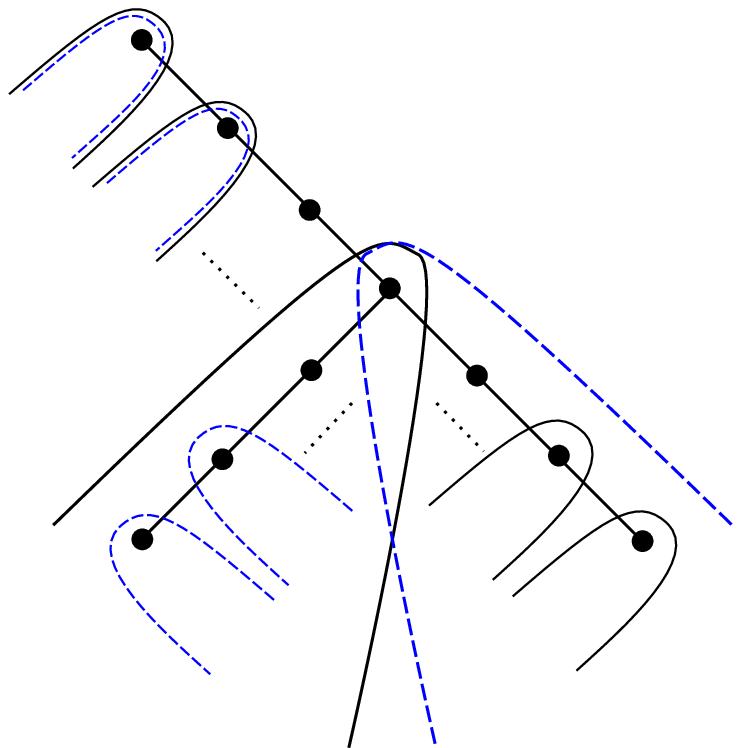
          \fi
          \caption{$M$- and $m$-decompositions of a tree $T$ ($M$-decomposition in black plain lines
          and $m$-decomposition in blue dashed lines).}
          \label{fig:Mm_dec}
    \end{figure}

    \begin{figure}[t]
        \[ \begin{array}{c|c}\begin{tikzpicture}
        [font=\scriptsize,level 1/.style={sibling distance=10mm}]
            \node [part3] {$2$};
            \node at (1,0) [root] {$3$};
            \node at (2,0) [part1] {$4$}
                 child{node[part3] {$7$}};
            \node at (3.5,0) [part4] {$5$}
                 child{node[root] {$6$}}
                 child{node[part4] {$11$}};
            \node[part1] at (5,0) {$8$}
                 child{node[part2] {$9$}};
            \node[part1] at (6,0) {$10$};
    \end{tikzpicture}&
        \begin{tikzpicture}
            [font=\scriptsize]
    \node[part3] (pr) at (0,0) {$\pi_3$ }; 
    \node[part1] (p1) at (0.75,.75) {$\pi_5$ };
    \node[part2] (p2) at (1.5,0) {$\pi_4$ };
    \node[root] (p3) at (3.5,0) {$\pi_2$ }; 
    \node[part4] (p4) at (2.75,.75) {$\pi_6$ };
    \draw[dashed,->]  (p3) -- (p4);
    \draw[dashed,->]  (p2) -- (p1);
    \draw[dashed,->]  (pr) -- (p1);
    \draw[dashed,->]  (p4) {} .. controls +(-1,1) and +(1,1) .. (p4);
        \end{tikzpicture}
    \end{array}
    \]
        \caption{The $10$-decomposition and $10$-dependence graph of the tree $T$ from 
        Figure~\ref{fig:ex_splice_irreducible}.}
        \label{fig:10-dec}
        \label{fig:10-dep-graph}
    \end{figure}
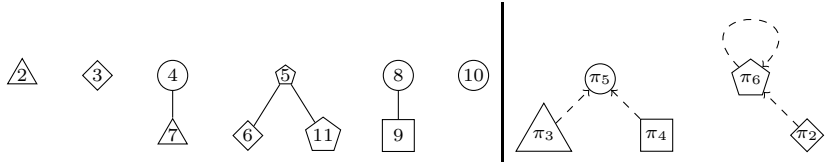

We now describe partially the component graph $G_m(T)$. 
To help the reader, an example is given in Figure \ref{fig:10-dep-graph}.

The vertex $M$ lies in the subtree $T^{(u)}_m$, so there is an edge from $\pi^M$ to $\pi^u$ in $G_m(T)$. Now $u\neq M$, and $u$ cannot be the maximum element of any other block of $\pi$, since this would imply that $\pi^u$ has outdegree $0$ in $G_M(T)$, which would contradict Lemma~\ref{lem_conntree}. 

Now consider the maximum element in $\pi^u$, that we will denote by $\mu^u$. Then $\mu^u\neq u$ is a descendant of $u$ in $T$, and cannot be contained in $T^{(u)}_M$, since that would create a loop in $G_M(T)$, again contradicting Lemma~\ref{lem_conntree}. Thus $\mu^u$ is contained in $T^{(u)}_m$ (as a nonroot vertex), which implies that there is a loop at $\pi^u$ in $G_m(T)$.

But clearly $\pi^m$ has outdegree $0$ in $G_m(T)$. Putting this together with the facts that there is a loop at $\pi^u$ in $G_m(T)$, and that each vertex has outdegree at most $1$, we see that $\pi^u$ and $\pi^m$ are contained in different components of $G_m(T)$, so $G_m(T)$ is not connected. Moreover, the edge from $\pi^M$ to $\pi^u$ in $G_m(T)$ implies that $\pi^M$ and $\pi^u$ are in the same component of $G_m(T)$ (note that we can have $\pi^u=\pi^M$).
    
Then Lemma~\ref{Lem_splcondvtx} implies that $T$ can be written uniquely as
\[T=\spl(T_1,m;T_2,t),\]
where $T_1$ and $T_2$ are $\pi$-increasing trees, $m$ is in $T_1$, $t$ and $M$ are in $T_2$, with $t<m$, and $T_1$ is irreducible: this is obtained by letting $V(T_1)$ be the vertex-set of the connected component of $G_m(T)$ that contains $\pi^m$, which means that $G_m(T_1)$ is connected. But the elements of $\pi^M$ are contained in $V(T_2)$, so $m$ is the maximum vertex in $T_1$, and so $T_1$ is irreducible.

It remains to prove that $T_2$ is also irreducible. To do this, we look at the $M$-decomposition of $T_2$ (since $M$ is the maximum vertex in $T_2$), which is obtained from the $M$-decomposition of $T$ as follows
(see the proof of Lemma \ref{lem:splice_reducible}):
    \begin{itemize}
        \item delete the blocks $T^{(a_i)}_M$, for which $a_i$ belongs
            to $T_1$.
            Since $T^{(u)}_m$ is in $T_2$, this can happen only for $i <j$,
            that is for blocks before $T^{(u)}_M$ in the decomposition;
        \item replace the block $T^{(u)}_M$ by some subblock  $(T_2)^{(u)}_M$
            still rooted at $u$.
    \end{itemize}
    In particular, if a block $X$ of $\pi$ is in $V(T_1)$ and
    if there is an edge from $Y$ to $X$ in $\G{M}{T}$,
    then $Y$ is also in $T_1$.
    This means that, if a vertex $X$ is deleted when going from $\G{M}{T}$
    to $\G{M}{T_2}$, all vertices pointed to it are also deleted,
    and recursively.
    For vertices that are not deleted their outgoing edge is not modified.

    Hence, since $\G{M}{T}$ is a directed tree (by Lemma \ref{lem_conntree}),
    $\G{M}{T_2}$ is also a directed tree, which implies that $T_2$
    is irreducible and ends the proof of the lemma.
\end{proof}

\begin{example}
    As (a) is quite easy, we only give here an example of (b).
    Consider the graph $T$ from Figure \ref{fig:ex_splice_irreducible}.
    Since it is irreducible, it can be written uniquely as 
    $$T=\spl(T_1,m;T_2,t),$$
    with the conditions given in Lemma \ref{lemsplirred} (b).
    This decomposition is the one from Figure \ref{fig:ex_splice_irreducible}.
    Note that the parts $\pi_2$ and $\pi_6$, which are the ones included in the vertex-set of $T_2$,
    correspond to the vertices in the connected component of $\pi_6$ in $G_{10}(T)$
    (see Figure \ref{fig:10-dep-graph}), as explained in our proof.
\end{example}

We now record a final straightforward fact about $\pi$-increasing trees.

\begin{proposition}\label{oneblock}
Suppose that $T$ is a $\pi$-increasing tree in which the vertex-set consists of a single block of $\pi$. Then $T$  is uniquely the increasing chain consisting of the elements of that block.
\end{proposition}

Now we are ready to prove Theorem~\ref{thm_bijn}.
\vspace{.1in}

\noindent
\textit{Proof of Theorem~\ref{thm_bijn}.} Suppose $r\ge 2$ and $\pi\in\Pi_1(\N_r)$, where $\pi$ has $k\ge 2$ blocks, and consider an arbitrary tree $T\in\mcE(\pi)$. Then $T$ is a $\pi$-increasing tree in which the vertex-set consists of two or more blocks of $\pi$ and contains $1$, and from Proposition~\ref{prop_Treducible}, $T$ is reducible. 

From Lemma~\ref{lemsplirred} (a), since $\mu_k$ is the maximum label,
$T$ can be written uniquely as
\[\spl(T_1,M;T_2,t)\]
where $T_1$ is irreducible.
Call $\tau_k=T_1$, $\tau_1=T_2$ and $\nu=t$.
We have uniquely reversed Stage $k$ in Construction~\ref{mapping} 
(by Proposition~\ref{proposition_psiandred}, in the input at Stage $k$,
$\tau_k$ is always irreducible).
Note that $1$ and $\nu$ lie in $\tau_1$ (from Lemma~\ref{lemsplirred} (a)).

If $k \ge 3$, we now want to invert Stage $k-1$, 
which is a splice in which the first splicing vertex is always $\mu_{k-1}$.
So we shall look at $\mu_{k-1}$ and consider two cases:
\begin{description}
    \item[Case 1.] $\mu_{k-1}$ lies in $\tau_1$.
        In this case, from Proposition~\ref{prop_Treducible}, $\tau_1$ is reducible.
        In addition, since $\mu_k$ (and all vertices in block $\pi_k$) lie in $\tau_k$,
        $\mu_{k-1}$ is the maximum of $\tau_1$.
        Thus, from Lemma~\ref{lemsplirred} (a), $\tau_1$ can be written uniquely as
        \[\spl(T_1,\mu_{k-1};T_2,t)\]
        with $T_1$ irreducible.
        Call $\tau_{k-1}=T_1$, $\tau_1=T_2$, $c_{k-1}=\nu$ and 
        then update the value of $\nu$ to $t$.
        Then $c_{k-1}$ lies in $\tau_1$ or $\tau_{k-1}$ 
        (since, before this step, $\nu$ lies in $\tau_1$).
        Recall that, in the input of Stage $k-1$, $\tau_{k-1}$ is always irreducible 
        (Proposition~\ref{proposition_psiandred}).
        Thus, we have uniquely reversed Stage $k-1$ in Construction~\ref{mapping},
        which was an internal splice.

        Note that, after this step, $1$ and $\nu$ still lie in $\tau_1$.\smallskip

    \item[Case 2.] $\mu_{k-1}$ lies in $\tau_k$.
        In this case, $\mu_{k-1}$ is the second maximum of $\tau_k$, and recall
        that $\tau_k$ is irreducible by construction.
        From Lemma~\ref{lemsplirred} (b), $\tau_k$ can be written uniquely as
        \[\spl(T_1,\mu_{k-1};T_2,t)\]
        with $T_1$, $T_2$ irreducible, where $M$ lies in $T_2$ and $t < \mu_{k-1}$.
        Call $\tau_{k-1}=T_1$, $\tau_k=T_2$ and $c_{k-1}=t$
        (the value of $\nu$ is unchanged).
        Then $c_{k-1}$ lies in $\tau_k$.
        Recall that, in the input of Stage $k-1$, $\tau_{k-1}$ and $\tau_k$ 
        are always irreducible 
        and $M$ lies in $\tau_k$ (Proposition~\ref{proposition_psiandred}).
        Thus, we have uniquely reversed Stage $k-1$ in Construction~\ref{mapping},
        which was an external splice.

        Note that, since $\nu$ and $\tau_1$ have not been changed, $1$ and $\nu$ still lies in $\tau_1$.\smallskip
\end{description}

Now, all remaining stages $k-2, \ld ,2$ of Construction~\ref{mapping} can be uniquely reversed exactly 
as for Stage $k-1$ (Case 2 in general is ``$\mu_{i}$ lies in $\tau_{i+1}$, \dots ,$\tau_k$'').
After reversing Stage $i$, we have trees $\tau_1$, $\tau_i$, \ldots, $\tau_k$,
such that $\tau_1$ contains $1$, and $\nu$ and $\tau_i$, \ldots, $\tau_k$
are irreducible and contain $\mu_i$, \ldots, $\mu_k$, respectively.
From Proposition~\ref{oneblock}, we recover at the end 
the initial forest with components $\tau_{\ell}$, $\ell =1,\ld ,k$, 
where $\tau_{\ell}$ is the increasing chain consisting of the elements
of the block $\pi_{\ell}$ of $\pi$.
Along the way, we recover uniquely the elements $c_i$ of the $(k-2)$-tuple $c\in\mcC(\pi)$.
We conclude that $c=\psi_{\pi}^{-1}(T)$, and that $\psi_{\pi}$ is a bijection.
$\hfill\Box$

\bibliographystyle{abbrv}
\bibliography{2013May}

\end{document}